\documentclass[11pt]{amsart} 
\usepackage{amssymb,amsmath,latexsym,enumerate,graphicx,bbm,mathptmx,microtype,cite,ifthen,color,tikz}
\allowdisplaybreaks
\usepackage{cmbright}

\usepackage{xcolor}
\usepackage{pgfplots}
\usepackage[T1]{fontenc}
\usepackage{braket}
\usepackage{enumitem}
\usetikzlibrary{shapes.geometric}

\usetikzlibrary{arrows.meta} 
\usetikzlibrary{decorations.markings} 

\newtheorem{theorem}{Theorem} 
\newtheorem{corollary}[theorem]{Corollary}

\newtheorem{exam}{Example}

\newtheorem*{rem}{Remarks}

\newcommand\commentout[1]{}
\newcommand\Def[1]{{\bf #1}}

\newcommand\Asc{\operatorname{Asc}} 
\newcommand\Des{\operatorname{Des}} 
\newcommand\des{\operatorname{des}} 
\newcommand\maj{\operatorname{maj}}

\newcommand\JH{\operatorname{JH}}

\newcommand\ZZ{\mathbb{Z}}

\makeatletter 
\newtheorem*{rep@theorem}{\rep@title}\newcommand{\newreptheorem}[2]{%
\newenvironment{rep#1}[1]{%
\def\rep@title{\bf #2 \ref{##1}}%
\begin{rep@theorem}}%
{\end{rep@theorem}}}
\makeatother
\newreptheorem{theorem}{Theorem}

\newcounter{teach}
\begin{document}

\title{MacMahon's Double Vision: Partition Diamonds Revisited}


\author{Matthias Beck}
\author{Kobe Wijesekera}
\address{Department of Mathematics\\ San Francisco State University}
\email{[mattbeck,kwijesekera]@sfsu.edu}


\begin{abstract}
Plane partition diamonds were introduced by Andrews, Paule, and Riese (2001) as part of their study of
MacMahon's $\Omega$-operator in search for integer partition identities. More recently, Dockery, Jameson,
Sellers, and Wilson (2024) extended this concept to $d$-fold partition diamonds and found their generating
function in a recursive form. We approach $d$-fold partition diamonds via Stanley's (1972) theory of $P$-partitions and give a closed formula for a bivariate generalization of the Dockery--Jameson--Sellers--Wilson generating function; its main ingredient is the Euler--Mahonian polynomial encoding descent statistics of permutations.
\end{abstract}



\date{7 Jun 2025}

\maketitle


\section{Introduction}

\begin{figure}[h]
\centering
\begin{tikzpicture}
\begin{scope}[every node/.style={circle,draw,scale=0.75,inner sep=2pt}]
    \node[label={$a_1$}] (a0) at (0,0) {};
    \node[label={$a_2$}] (b1) at (1,1) {};
    \node[label=below:{$a_3$}] (b2) at (1,-1) {};
    \node[label={$a_4$}] (a1) at (2,0) {};
    \node[label={$a_5$}] (b3) at (3,1) {};
    \node[label=below:{$a_6$}] (b4) at (3,-1) {};
    \node[label={$a_7$}] (a2) at (4,0) {};
    \node[label={$a_8$}] (b5) at (5,1) {};
    \node[label=below:{$a_9$}] (b6) at (5,-1) {};
    \node[label={$a_{10}$}] (a3) at (6,0) {};
\end{scope}

\begin{scope}[>={Stealth[black]}, decoration={
    markings,
    mark=at position 0.6 with {\arrow{Latex}}}
    ] 
    \draw[postaction={decorate}] (a0) -- (b1);
    \draw[postaction={decorate}] (a0) -- (b2);
    \draw[postaction={decorate}] (b1) -- (a1);
    \draw[postaction={decorate}] (b2) -- (a1);
    \draw[postaction={decorate}] (a1) -- (b3);
    \draw[postaction={decorate}] (a1) -- (b4);
    \draw[postaction={decorate}] (b3) -- (a2);
    \draw[postaction={decorate}] (b4) -- (a2);
    \draw[postaction={decorate}] (a2) -- (b5);
    \draw[postaction={decorate}] (a2) -- (b6);
    \draw[postaction={decorate}] (b5) -- (a3);
    \draw[postaction={decorate}] (b6) -- (a3) node[right] {$\cdots$};
\end{scope}
\end{tikzpicture}
\caption{A plane partition diamond.}
\label{fig:2fold}
\end{figure}
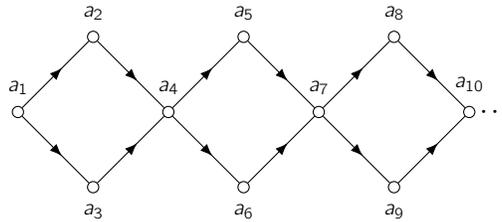

A \Def{plane partition diamond} is an integer partition $a_1 + a_2 + \dots + a_k$ whose parts satisfy the
inequalities given by Figure~\ref{fig:2fold}, where each directed edge represents $\ge$.
Plane partition diamonds were introduced by Andrews, Paule, and Riese~\cite{andrewspauleriesePA8}, who found their generating function as
\[
  \prod_{ n \ge 1 } \frac{ 1 + q^{ 3n-1 } }{ 1 - q^n } \, .
\]
They proved this result as a part of an impressive series of papers on partition identities via MacMahon's
$\Omega$-operator; indeed MacMahon himself used it for the case of a single $\diamond$~\cite[Volume~2,
Section~IX, Chapter~II]{macmahon}. 

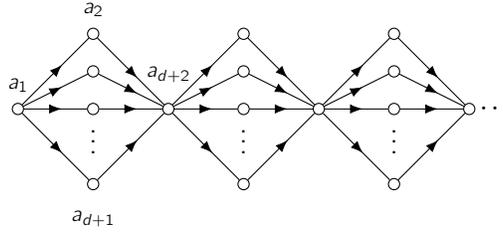
\begin{figure}[h]
\centering
\begin{tikzpicture}
\begin{scope}[every node/.style={circle,draw,scale=0.75,inner sep=2pt}]
\foreach \i/\k in {0/0,2/1,4/2,6/3}
    {\node[] (a\k) at (\i,0) {};}
\foreach \i in {1,2,3}
    {\node[] (b1\i) at (1,-0.5*\i+1.5) {};
     \node[] (b2\i) at (3,-0.5*\i+1.5) {};
     \node[] (b3\i) at (5,-0.5*\i+1.5) {};
     \node[] (b\i d) at (2*\i-1,-1) {};}
\foreach \i in {1,2,3}
    {\node also [] (b\i 1);
     \node also [] (b\i d);}

\node also [label = $a_1$] (a0);
\node also [label = $a_2$] (b11);
\node also [label=below: $a_{d+1}$] (b1d);
\node also [label = $a_{d+2}$] (a1);
\end{scope}

\begin{scope}[>={Stealth[black]}, decoration={
    markings,
    mark=at position 0.6 with {\arrow{Latex}}}]
\foreach \i/\k in {0/1,1/2,2/3}
    {\foreach \j in {1,2,3,d}
    {\draw[postaction=decorate] (a\i) -- (b\k\j);
     \draw[postaction=decorate] (b\k\j) -- (a\k);
    }}
\foreach \i in {1,3,5}
    {\node[] at (\i,-0.33) {$\vdots$};}
\node[right] at (6,0) {$\cdots$};
\end{scope}
\end{tikzpicture}
\caption{A $d$-fold partition diamond.}
\label{fig:dfold}
\end{figure}

Dockery, Jameson, Sellers, and Wilson~\cite{dockeryjamesonsellerswilson} recently generalized the above concept to a \Def{$d$-fold partition diamond}, whose parts now follow the inequalities stipulated by Figure~\ref{fig:dfold}.
They proved that their generating function equals
\[
  \prod_{ n \ge 1 } \frac{ F_d ( q^{ (n-1)(d+1) + 1 } , q ) }{ 1 - q^n }
\]
where $F_d(q_0, w) \in \ZZ[q_0, w]$ is recursively defined via $F_1(q_0, w) = 1$ and
\begin{equation}\label{eq:recursion}
  F_d(q_0, w) \, = \, \frac{ (1-q_0 w^d) \, F_{ d-1 }(q_0, w) - w(1-q_0) \, F_{ d-1 }(q_0 w, w) }{ 1-w } \, .
\end{equation}
Once more the proof in~\cite{dockeryjamesonsellerswilson} uses MacMahon's $\Omega$-operator.

Our goal is to view the above results via Stanley's theory of $P$-partitions~\cite{stanleythesis,stanleyec1}.
Our first result gives a closed formula for the generating function for $d$-fold partition diamonds.
In a charming twist of fate, its main ingredient turns out to be the \Def{Euler--Mahonian polynomial}
\[
  E_d(x, y) \, := \, \sum_{ \tau \in S_d } x^{ \des(\tau) } y^{ \maj(\tau) } 
\]
which first appeared in a completely separate part of MacMahon's vast body of work~\cite[Volume~2, Chapter~IV,
Section~462]{macmahon};
here $\Des(\tau) := \{ j : \, \tau(j) > \tau(j+1) \}$ records the descent positions of a given permutation $\tau
\in S_d$, with the statistics $\des(\tau) := |\Des(\tau)|$ and $\maj(\tau) := \sum_{ j \in \Des(\tau) } j$.

\begin{theorem}\label{thm:firstthm}
The Dockery--Jameson--Sellers--Wilson polynomial $F_d(x,y)$ equals the Euler--Mahonian polynomial~$E_d(x,y)$.
\end{theorem}

This theorem consequently implies that~\eqref{eq:recursion} defines the Euler--Mahonian polynomials recursively.
We suspect that this recursion is known but could not find it in the literature.

\begin{figure}[h]
\centering
\begin{tikzpicture}
\begin{scope}[every node/.style={circle,draw,scale=0.75,inner sep=2pt}]
\foreach \i/\k in {0/0,2/1,4/2,5/n-1,7/n}
    {\node[] (a\k) at (\i,0) {};}
\foreach \i in {1,2,3}
    {\node[] (b1\i) at (1,-0.5*\i+1.5) {};
     \node[] (b2\i) at (3,-0.5*\i+1.5) {};
     \node[] (bn\i) at (6,-0.5*\i+1.5) {};}
\foreach \i/\k in {1/1,2/2,3.5/n}
    { \node[] (b\k d) at (2*\i-1,-1) {};}
\foreach \i in {1,2,n}
    {\node also [] (b\i 1);
     \node also [] (b\i d);}
\foreach \k in {0,1,2,n}{
\node also [] (a\k);
}
\foreach \k in {n-1}{\node also [] (a\k);}

\node also [label =left: $a_1$] (a0);
\node also [label = $a_2$] (b11);
\node also [label=below: $a_{d+1}$] (b1d);
\node also [label = $a_{d+2}$] (a1);
\node also [label=right: $a_c$] (an);

\end{scope}
\begin{scope}[>={Stealth[black]}, decoration={
    markings,
    mark=at position 0.6 with {\arrow{Latex}}}]
\foreach \i/\k in {0/1,1/2,n-1/n}
    {\foreach \j in {1,2,3,d}
    {\draw[postaction=decorate] (a\i) -- (b\k\j);
     \draw[postaction=decorate] (b\k\j) -- (a\k);
    }}
\foreach \i in {1,3,6}
    {\node[] at (\i,-0.33) {$\vdots$};}
\node[] at (4.5,0) {$\cdots$};   
\end{scope}
\end{tikzpicture}
\caption{The $d$-fold partition diamond poset of length $M$, with $c = M(d+1)+1$.}
\label{fig:dfoldfin}
\end{figure}
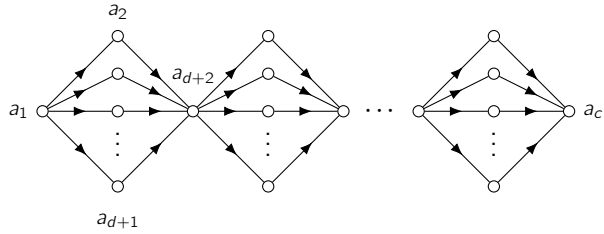

Theorem~\ref{thm:firstthm} is actually a corollary of our main result, which gives a 2-variable refinement as follows.
Let $M$ be the number of $\diamond$s in a finite version of the $d$-fold partition diamond depicted in
Figure~\ref{fig:dfoldfin}; we call $M$ the \Def{length} of the diamond.
We define $\sigma_{d,M}(a,b)$ be the generating function of $d$-fold partition diamonds of length $M$, where $a$
encodes the parts in the ``folds'' of the diamond, and $b$ encodes the parts in the links connecting the
$\diamond$s. That is,
\begin{equation}\label{eq:sigmadMdef}
  \sigma_{d,M}(a,b) \, := \, \sum a^{ a_2 + \dots + a_{ d+1 } + a_{ d+3 } + \dots + a_{ 2d + 2 } + a_{ 2d + 4 } + \dots + a_{ M(d+1) } } \, b^{ a_1 + a_{ d+2 } + \dots + a_{ M(d+1) + 1 } }
\end{equation}
where the sum is over all $d$-fold partition diamonds $a_1 + a_2 + \dots + a_{ M(d+1) + 1 }$.

\begin{theorem}\label{thm:main}
\[
  \sigma_{d,M}(a,b) \, = \, \frac{ \prod_{ n=1 }^M E_d ( a^{ (n-1)d } b^n , a ) }{ \left( 1 - a^{ Md } b^{ M+1 } \right) \prod_{ n=1 }^M \prod_{ j=0 }^d \left( 1 - a^{ nd-j } b^n \right) } \, .
\]
\end{theorem}

Naturally, Theorem~\ref{thm:firstthm} follows with $a = b = q$ and $M \to \infty$.
Another special evaluation ($a = 1$) gives the generating function, already discovered
in~\cite{dockeryjamesonsellerswilson}, of \Def{Schmidt type $d$-fold partition diamonds}, in which we sum only
the links $a_1 + a_{ d+2 } + a_{ 2(d+1) + 1 } + \cdots$.

\begin{corollary}[Dockery--Jameson--Sellers--Wilson]
The generating function for Schmidt type $d$-fold partition diamonds is given by
\[
  \prod_{ n \ge 1 } \frac{ E_d( q^n , 1) }{ (1-q^n)^{ d+1 } } \, .
\]
\end{corollary}
The polynomial $E_d(x,1) = \sum_{ \tau \in S_d } x^{ \des(\tau) }$ is known as an \Def{Eulerian polynomial}.

Section~\ref{sec:proof} contains our proof of Theorem~\ref{thm:main}.
As we will see, it can be applied to more general situations, e.g., allowing folds within the diamond of
different heights. We will outline this in Section~\ref{sec:final}.


\section{Proofs}\label{sec:proof}

We now briefly review Stanley's theory of $P$-partitions~\cite{stanleythesis,stanleyec1}.
Fix a finite partially ordered set $(P, \preceq)$. We may assume that $P = [c] := \{ 1, 2,
\dots, c \}$ and that $j \preceq k$ implies $j \le k$; i.e., $(P, \preceq)$ is \Def{naturally labelled}.
A \Def{linear extension} of $(P, \preceq)$ is an order-preserving bijection $(P, \preceq) \to ([c], \le)$.
It is a short step to think about a linear extension as a permutation in $S_c$; accordingly we define the
\Def{Jordan--H\"older set} of $(P, \preceq)$ as
\[
  \JH(P, \preceq) \, := \, \left\{ \tau \in S_c : \ \preceq_\tau \text{refines} \preceq \right\} 
\]
where $\preceq_\tau$ refers to the (total) order given by the linear extension corresponding to $\tau$.

A \Def{$P$-partition} is a composition $m_1 + m_2 + \dots + m_c$ such that $m: P \to \mathbb{Z}_{ \ge 0 }$ is
order preserving:\footnote{
Stanley defines $P$-partitions in an order-\emph{reversing} fashion.
}
\[
  j \preceq k \qquad \Longrightarrow \qquad m_j \leq m_k \, .
\]
It comes with the multivariate generating function
\[
  \sigma_P(z_1, z_2, \dots, z_c) \, := \, \sum z_1^{ m_1 } z_2^{ m_2 } \cdots z_c^{ m_c }
\]
where the sum is over all $P$-partitions.
The standard $q$-series for the $P$-partitions is, naturally, the specialization $\sigma_P(q, q, \dots, q)$.

One of Stanley's fundamental results, given here in the form of \cite[Corollary~6.4.4]{crt}, is that
\begin{equation}\label{eq:stanleygenfct}
  \sigma_P(z_1, z_2, \dots, z_c)
  \, = \sum_{ \tau \in \JH(P, \preceq) } \frac{ \prod_{ j \in \Des(\tau) } z_{ \tau(j+1) } z_{ \tau(j+2) }
\cdots z_{ \tau(c) } }{ \prod_{ j=0 }^{ c-1 } \left( 1 - z_{ \tau(j+1) } z_{ \tau(j+2) } \cdots z_{ \tau(c) }
\right) } \, .
\end{equation}
We will apply \eqref{eq:stanleygenfct} to the poset depicted in Figure~\ref{fig:dfoldfin}. 
This poset has a natural additive structure, and so we first review how to compute~\eqref{eq:stanleygenfct} for
\[ P \, = \, Q_0 \oplus Q_1 \oplus \dots \oplus Q_M \, , \] where the linear sum $A \oplus B$ of two posets $A$ and $B$ is defined on the ground set $A \uplus B$, with the relations inherited among elements of $A$ and those of $B$, together with $a \preceq b$ for any $a \in A$ and $b \in B$. 
Assuming that $Q_j$ has ground set $[q_j]$, let $s_0 := 0$ and $s_j := q_0 + q_1 + \dots + q_{ j-1 }$. 
Each element $\tau \in \JH(P)$ is uniquely given via
\[
\tau(j) =
    \begin{cases}
        \tau_0(j) & \text{if } j \in Q_0,\\
        \tau_1(j-s_1) + s_1 & \text{if } j \in Q_1', \\
        \tau_2(j-s_2) + s_2 & \text{if } j \in Q_2',\\
        \qquad \vdots & \\
        \tau_M(j-s_{M}) + s_{M} & \text{if } j \in Q'_M,
    \end{cases}
\]
for some $\tau_j \in \text{JH}(Q_j)$, where $0 \le j \le M$.
Here
$
  Q'_k \, := \, \left\{s_k + 1, s_k+2,\dots,s_k + q_k \right\} ,
$
with the relations induced by those in~$Q_k$.
Subsequently,
\[
 \Des(\tau) \, = \, \biguplus_{k=0}^M \left\{j + s_k \, : \, j \in \Des(\tau_k)\right\}
\]
and so the numerator in~\eqref{eq:stanleygenfct} becomes
\begin{align}
  &\prod_{ j \in \Des(\tau) } z_{ \tau(j+1) } z_{ \tau(j+2) } \cdots z_{ \tau(c) } \nonumber \\
  &\qquad = \, \prod_{k=0}^{M} \prod_{j \in \Des(\tau_k)} z_{\tau_k(j+1)+s_k} z_{\tau_k(j+2)+s_k} \cdots
z_{\tau_k(q_k)+s_k} z_{\tau_{k+1}(1)+s_{k+1}} \cdots z_{\tau_M(q_M)+s_M} \, , \label{eq:numerator}
\end{align}
with an analogous form for the denominator.
We now apply these concepts to the partition diamond poset $P$ in Figure~\ref{fig:dfoldfin}.

\begin{proof}[Proof of Theorem~\ref{thm:main}]
We have
\[
  P \, = \, \{ 1 \} \oplus \underbrace{ Q_d \oplus Q_d \oplus \dots \oplus Q_d }_{ M \text{ copies} } \, ,
\]
where $Q_d$ is the poset in
Figure~\ref{fig:dk}, an antichain with $d$ elements plus one more element $\succ$ the others. 
\begin{figure}[ht]
    \centering
    \begin{tikzpicture}
        \begin{scope}[every node/.style={circle,draw,scale=0.75,inner sep=2pt}]
            \node[label=right:$d+1$] (a1) at (0,0) {};

            \foreach \i in {1,2,3}
            {\node (b\i) at (-1,-0.5*\i+1.5) {};
            
            }
            \node[label=left: $d$] (bd) at (-1,-1) {};

            \node also [label=left:$1$] (b1);
            \node also [label=left:$2$] (b2);
            \node also [label=left:$3$] (b3);
        
        \end{scope}

        \begin{scope}[>={Stealth[black]}, decoration={
        markings,
        mark=at position 0.6 with {\arrow{Latex}}}
        ] 
            \node at (-1,-0.33) {$\vdots$};
            \draw [postaction=decorate] (b1) -- (a1);
            \draw [postaction=decorate] (b2) -- (a1);
            \draw [postaction=decorate] (b3) -- (a1);
            \draw [postaction=decorate] (bd) -- (a1);
        \end{scope}
    \end{tikzpicture}

    \caption{The poset $Q_d$.}
    \label{fig:dk}
\end{figure}
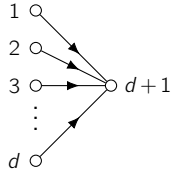

The generating function $\sigma_{ d, M}(a,b)$ defined in \eqref{eq:sigmadMdef} is the special evaluation
\[
  \sigma_{ d, M}(a,b) \, = \, \sigma_P(z_1, z_2, \dots, z_c)
  \qquad \text{ where } \qquad 
z_j = 
\begin{cases}
    a & \text{if } j \not\equiv 1 \bmod d+1,\\
    b & \text{if } j \equiv 1 \bmod d+1.
\end{cases}
\]
We now count how many $a$s and $b$s appear in~\eqref{eq:numerator}.
Each $\tau_k(d+1) = d+1$, contributing a factor of $b^{ M-k+1 }$ to~\eqref{eq:numerator}.
The remaining variables contribute $a^{ d-j }$ for the index $k$ and $a^{ (M-k)d }$ for the indices~$>k$.
The denominator in~\eqref{eq:stanleygenfct} is computed analogously; note that, unlike the numerator, it has a
contribution stemming from the minimal element in $P$.
In summary, this analysis yields
\begin{align}
  \sigma_{ d, M}(a,b) 
  \, &= \, \sum_{\tau \in \JH(P)}
\frac{\prod_{k=1}^M\prod_{j\in\Des(\tau_k)}a^{(M-k+1)d-j}b^{M-k+1}}{\left(1-a^{Md}b^{M+1}
\right)\prod_{k=1}^{M}\prod_{j=0}^{d} \left( 1-a^{(M-k+1)d-j}b^{M-k+1} \right)} \nonumber \\
     &= \, \frac{\prod_{k=1}^M \sum_{\tau \in S_d}\prod_{j\in\Des(\tau)} a^{(M-k+1)d-j}b^{M-k+1}
}{\left(1-a^{Md}b^{M+1} \right)\prod_{k=1}^{M}\prod_{j=0}^{d} \left( 1-a^{(M-k+1)d-j}b^{M-k+1} \right)} \, ,
\label{eq:tobesimplified}
\end{align}
where the second equation follows from the fact that each $\tau_k$ stemming from some $\tau \in \JH(P)$ fixes
$d+1$, but can freely permute the remaining $d$ elements.

By standard bijective arguments,
\begin{align*}
  \sum_{\tau \in S_d}\prod_{j\in\Des(\tau)} a^{(M-k+1)d-j}b^{M-k+1}
  \, &= \, \sum_{\tau \in S_d}\prod_{j\in\Asc(\tau)} a^{(M-k+1)d-j}b^{M-k+1} \\
  \, &= \, \sum_{\tau \in S_d}\prod_{d-j\in\Asc(\tau)} a^{(M-k)d+j}b^{M-k+1} \\
  \, &= \, \sum_{\tau \in S_d}\prod_{j\in\Des(\tau)} a^{(M-k)d+j}b^{M-k+1} \, ,
\end{align*}
where $\Asc(\tau) := \{ j : \, \tau(j) < \tau(j+1) \}$.
Substituting this back into~\eqref{eq:tobesimplified} and making the change of parameters $n:= M-k$ gives
\begin{align*}
  \sigma_{ d, M}(a,b) 
  \, &= \, \frac{\prod_{n=0}^{M-1} \sum_{\tau \in S_d} \prod_{j \in \Des(\tau)} a^{nd+j}b^{n+1} } {\left(1-a^{Md}b^{M+1} \right)\prod_{n=0}^{M-1}\prod_{j=0}^{d} \left( 1-a^{(n+1)d-j}b^{n+1} \right)} \\
     &= \, \frac{\prod_{n=1}^{M} \sum_{\tau \in S_d} \prod_{j \in \Des(\tau)} a^{(n-1)d+j}b^{n} } {\left(1-a^{Md}b^{M+1} \right)\prod_{n=1}^{M}\prod_{j=0}^{d} \left( 1-a^{nd-j}b^{n} \right)} \\
     &= \, \frac{\prod_{n=1}^{M} \sum_{\tau \in S_d} a^{ \maj(\tau) } \left( a^{(n-1)d}b^{n} \right)^{
\des(\tau) } } {\left(1-a^{Md}b^{M+1} \right)\prod_{n=1}^{M}\prod_{j=0}^{d} \left( 1-a^{nd-j}b^{n} \right)} \\
     &= \, \frac{ \prod_{ n=1 }^M E_d ( a^{ (n-1)d } b^n , a ) }{ \left( 1 - a^{ Md } b^{ M+1 } \right) \prod_{
n=1 }^M \prod_{ j=0 }^d \left( 1 - a^{ nd-j } b^n \right) } \, . \qedhere
\end{align*}
\end{proof}


\section{Another Extension}\label{sec:final}

The \emph{ansatz} for our proof of Theorem~\ref{thm:main} is, naturally, amenable to more general constructs. We
give one sample here, whose proof is analogous to that of Theorem~\ref{thm:main}.

Let $\{d_j\}_{j=1}^M$ be a finite sequence of positive integers. We define the \Def{multifold partition diamond}
corresponding to this sequence to be a partition whose parts follow a similar structure as those in Figure
\ref{fig:dfoldfin}, but the $j$th diamond has $d_j$ folds.    
The accompanying bivariate generating function is $\sigma_{ d_1, \dots, d_M } (a,b)$, where again $a$
encodes the parts in the ``folds'' of the diamond, and $b$ encodes the parts in the links connecting the~$\diamond$s.
Let $\omega_k := \sum_{j=k+1}^M d_j$.

\begin{theorem}
\[
      \sigma_{ d_1, \dots, d_M } (a,b)
      \, = \, \frac{\prod_{k=1}^M E_{d_k}(a^{\omega_k}b^{M-k+1},a)}{\left( 1- a^{\omega_0} b^{M+1} \right)
\prod_{k=1}^M \prod_{j=0}^{d_k} \left( 1 - a^{\omega_k+d_k - j} b^{M-k+1}\right)} \, .
\]
\end{theorem}


\bibliographystyle{amsplain}
\bibliography{bib}

\setlength{\parskip}{0cm} 

\end{document}